\documentclass[
aip,jmp,
reprint,
a4paper,
longbibliography,
nofootinbib,
nobibnotes
]{revtex4-2}
\usepackage[utf8]{inputenc}
\usepackage[T1]{fontenc}

\usepackage{amsmath,amssymb,amsthm,amsfonts,mathtools}
\usepackage{dsfont}
\usepackage{braket}
\usepackage{comment}
\usepackage[colorlinks]{hyperref}
\usepackage{enumitem}
\usepackage[]{algorithm2e}
\usepackage[usenames,dvipsnames,table]{xcolor}
\usepackage{graphicx}
\usepackage{tikz}
\usetikzlibrary{matrix,arrows.meta}
\usetikzlibrary{positioning,fit,calc}
\tikzstyle{block} = [draw=black, thick, text width=2cm, minimum height=1cm, align=center]  
\tikzstyle{arrow} = [thick,->,>=stealth]
\usepackage{textcmds}
\usepackage{hyperref}
\usepackage{thmtools}
 
\def\tr{\operatorname{tr}}
\def\Tr#1{\text{tr}\left(#1\right)}

\def\sth{\text{s.th. }}
\def\mc#1{\mathcal{#1}}

\def\und{\text{ and }}%

\newif\ifbem
\bemtrue

\theoremstyle{definition}
\newtheorem{theorem}{Theorem}

\newtheorem{prop}{Proposition}
\newtheorem{example}[theorem]{Example}

\newtheorem{rem}[theorem]{Remark}
\newtheorem{lemma}[theorem]{Lemma}
\newtheorem{cor}[theorem]{Corollary}
\newtheorem{definition}[theorem]{Definition}

\let\openone\undefined

\newcommand\openone{\mathds{1}}
\def\cs{\mathcal{C}^\star}
\def\pp#1{\mathcal{P}_+(\mathcal{#1})}
\def\pps#1{\mathcal{P}_+(\mathcal{#1})^\star}
\def\sth{\text{s.th.}}

\hypersetup{
    colorlinks = true,
citecolor=CadetBlue,
linkcolor=MidnightBlue,
urlcolor=MidnightBlue
}

\begin{document}

\title{\texorpdfstring{Hierarchies for Semidefinite Optimization in $\cs$-Algebras}{}}

\author{Gereon Koßmann}
\email{kossmann@physik.rwth-aachen.de}
\affiliation{Institute for Quantum Information, RWTH Aachen University, Aachen, Germany}

\author{René Schwonnek}
\email{rene.schwonnek@itp.uni-hannover.de}
\affiliation{Naturwissenschaftlich–Technische Fakultät, Universität Siegen, Germany}
\affiliation{Institut f\"{u}r Theoretische Physik, Leibniz Universit\"{a}t Hannover, Germany}

\author{Jonathan Steinberg}
\email{steinberg@physik.uni-siegen.de}
\affiliation{Naturwissenschaftlich–Technische Fakultät, Universität Siegen, Germany}

\begin{abstract}
    Semidefinite Optimization has become a standard technique in the landscape of Mathematical Programming that has many applications in finite dimensional Quantum Information Theory. 
    This paper presents a way for finite-dimensional relaxations of general cone programs on $\cs$-algebras which have structurally similar properties to ordinary cone programs, only putting the notion of positivity at the core of optimization. We show that well-known hierarchies for generalized problems like NPA\cite{navascues2007} but also Lasserre's hierarchy\cite{lasserre} and to some extend symmetry reductions of generic SDPs by de-Klerk et al. \cite{de2007reduction} can be considered from a general point of view of $\cs$-algebras in combination to optimization problems.     
\end{abstract}

\begin{abstract}
    The class of Semidefinite Programs on matrices has a natural extension to conic optimization problems on $\cs$-algebras. Especially in the field of Quantum Information Theory many relevant optimization problems  can be  cast in this way. 
    Finding exact solutions to these problems is however a notoriously hard task. 

    In this paper we investigate a construction for finite-dimensional relaxations of general cone programs on $\cs$-algebras. 
    This construction gives outer bounds and can be grasped in terms of positive maps and basic linear algebra. 
    We show that well-known hierarchies like NPA\cite{navascues2007} but also Lasserre's hierarchy\cite{lasserre} and to some extend symmetry reductions of generic SDPs by de-Klerk et al. \cite{de2007reduction} can be formulated from this perspective.     
\end{abstract}
\maketitle
\section{Introduction}

In recent years, it has become increasingly apparent that many optimization problems can be represented as cone programs. In particular the community figured out that there is an incredible strong interplay between cone programs and convex optimization problems whereby the convex set of feasible points is not directly accessible\cite{lasserre, Berta2015, Doherty_2004}. In quantum information the most prominent representative of this type of problem is the optimization task over separable states. Optimizing over separable states is a convex optimization problem but the description of the set of separable states is computationally hard\cite{Doherty_2004}. However, the solution is a hierarchy of efficient SDPs and the complexity of the problem is translated to a convergent hierarchy of convex optimization problems. Even from another point of view, one can ask whether two parties in a quantum experiment can achieve certain probability distributions. This task was solved by the famous NPA hierarchy\cite{navascues2007,navascues2008}. What all these hierarchies have in common is that they map the structure of states to a finite space and have a kind of mapping of positive elements to positive elements, because in quantum information theory, a generic cone structure of positivity exists due to the state space of quantum systems. However, due to the existence of mixed states, more general forms of positivity have to be considered. 

In this work, we focus on this idea of pursuing positivity in the most general structure of $\cs$-algebras and assign a $\cs$-algebra to the description of a quantum experiment. In the simplest and yet also most general form, this is done by rules such as how two or more parties are spatially related\cite{algebra_generated_by_two}.  
This top down approach to the description of an experiment in nature inevitably leads to questions about the state space and properties of observables with respect to states. 
For example, if one wants to find the optimal state for a given observable over all states, this is directly a generalized form of a SDP: Optimize over all positive functionals so that they are normalized. Usually this problem is hard to tackle because the $\cs$-algebra to be considered is infinite dimensional. 

In this paper, we present a generalized formalism that is both mathematically sound with respect to the rules of the system (i.e., the spatial partitioning of the parties) and capable of solving general optimization problems.


The success and popularity can be mainly ascribed to two different points. First, a huge variety of problems can be formulated (at least approximately) in the framework of semidefinite programming. Second, there exists algorithms, like interior-point methods, which allow for an efficient computation of the solution. 
More formally, a SDP describes the task of optimizing a given linear functional over the cone of positive (semidefinite) matrices of a fixed dimension $n$, under linear constraints. 

In addition to our motivation in quantum information, we note in passing that many well-known problems can be formulated as $\cs$-SDP. In particular, many typical benchmark examples for quantum computing such as Boolean optimization can be understood in a general way as $\cs$-SDP.

In this work, we investigate a generalization of semidefinite programming from the algebra of matrices to an arbitrary $\cs$-algebra, which is obtained by replacing the cone of positive semidefinite matrices with the cone of positive semidefinite functionals. 


Further, we show that the finite dimensional approximations can be rephrased as a conventional matrix SDP, turning them into a efficient solvable problem. Subsequent, we discuss its relations to well known hieraries as those of Lasserre \cite{lasserre} and Navascues-Pironio-Acin (NPA) \cite{navascues2007}. \


\section{Preliminaries and introductory example}\label{sec:prelim_exam}
We start our considerations with a simple example coming from an operational point of view. 
Consider the following generic problem in quantum theory: $H \in \mathcal{B}(\mathcal{H}_1)$ is a (selfadjoint) bounded operator on a Hilbert space $\mathcal{H}_1$ and $\mathcal{B}_1(\mathcal{H}_1)^+$ is the set of all positive trace-class operators. We aim to minimize the expectation value of 
\begin{align}\label{prelim:opt_problem}
    \min &\tr[\rho H] \\
    \text{s.th.} &\tr[\rho] = 1 \\
    \rho \in &\mathcal{B}_1(\mathcal{H}_1)^+.
\end{align}
This problem is inherently hard, because $H$ is in general an operator on a infinite dimensional Hilbert space and apparently there is a priori no efficient way to handle the set of states. 

If we now assume to have access to a quantum channel\footnote{i.e. a completely positive and trace-preserving linear map.} 
\begin{align*}
    T:\mathcal{B}_1(\mathcal{H}_1)\to \mathcal{B}_1(\mathcal{H}_2)
\end{align*}
and we know an operator $M \in \mathcal{B}(\mathcal{H}_2)$ such that $T^\star(M) = H$, then it follows that the optimization problem \eqref{prelim:opt_problem} is equivalent to\footnote{Recall that if the channel $T$ is trace-preserving the dual mapping is unital.}
\begin{align*}
    \min &\tr[\sigma M] \\
    \text{s.th.} &\tr[\sigma] = 1 \\
    \sigma \in T(&\mathcal{B}_1(\mathcal{H}_1)^+) \subset \mathcal{B}_1(\mathcal{H}_2)^+.
\end{align*}
Now it can happen that the Hilbert space $\mathcal{H}_2$ is even finite dimensional. We have shifted the entire difficulty of the problem to the description of $T(\mathcal{B}_1(\mathcal{H}_1)^+)$, i.e. the image of the quantum channel $T$. Nevertheless, this set contains global information of the cone $\mathcal{B}_1(\mathcal{H}_1)^+$ and is therefore in general hard to describe. However it could be contained in a finite dimensional space. We can make an illustrative example: consider an entanglement breaking channel, which can be given by\cite{Horodecki_2003, holevo2005separability}
\begin{align*}
    \rho \mapsto \sum_k \sigma_k \tr[F_k \rho]
\end{align*}
for density operators $(\sigma_k)_k \in \mathcal{B}_1(\mathcal{H}_2)^+$ and a POVM $\{F_k\}_k \subset \mathcal{B}(\mathcal{H}_1)$. It is clear that the dual mapping is given by
\begin{align*}
    M \mapsto \sum_k \tr[\sigma_k M] F_k.
\end{align*}
Therefore, if we assume that $\{F_k\}_k$ corresponds to the spectral decomposition of $H$ and $M$ is chosen in a way that $\tr[\sigma_k M]$ is the $k$th eigenvalue, then we can conclude that this is in general an infinite series and therefore the cone of interesting density operators 
\begin{align*}
    \text{cone}(\sigma_k \ \vert \ k  \in \mathbb{N})
\end{align*}
is hard to describe although it is contained in a possible finite dimensional space. 
This example provides us with some nice theoretical insights. First of all it tells us how to convert at least theoretically infinite dimensional optimization problems to finite dimensional spaces. Moreover if we can approximate $T(\mathcal{B}_1(\mathcal{H}_1)^+)$ appropriately with linear constraints, then we can relax the problem with a standard semidefinite program. 
As it turns out, in the example above we have chosen the state space as the set of trace-class operators, which would correspond to the normal states. In order to avoid this discussion about duality, we abstract to general $\cs$-algebras and consider the dual as the generic state space. In addition, as we will conclude, our considerations are not restricted to the special class of quantum channels. Positivity of linear maps will be the central property of this paper in comparison to completely positivity and trace-preserving maps.

The structure of the document is as follows: in \autoref{sec:sem_optincstar} we introduce into our  notion of $\cs$-SDP. In \autoref{sec:finite_dim_relax} we show how to relax general $\cs$-SDP and in \autoref{sec:convergence} we show a convergence result for particular assumptions. In the last section \autoref{sec:symmetry_red} we give a small overview about the connection to symmetry reductions in this context.

\section{\texorpdfstring{Semidefinite optimization on $\cs$-algebras}{}}\label{sec:sem_optincstar}

We generalize now our toy example from \autoref{sec:prelim_exam} to a general language for so called \emph{$\cs$-SDP} and define what we mean by optimality in this context. 

\subsection{Conventional SDPs}\label{subsection:conventional_sdp}
Formally, an instance of a conventional SDP is given by a set of $m+1$ selfadjoint matrices $\{F_0,F_1,\dots,F_m\}$, all coming from a common space $\mathbb{C}^{n\times n}$, and a set of real numbers $\{f_1,\dots,f_m\}$ that together define the optimization problem
\begin{align}
\label{sdpdefmat}
    &&\min \quad&\Tr{\rho F_0} &&\qquad\\
    &&\sth \quad&\Tr{\rho F_i}\leq f_i &&\forall i\in\{1,\dots,m\}\nonumber\\
    &&\quad \quad&\rho \geq 0&&\qquad
    \nonumber,
\end{align}
which is conventionally referred to as the primal form of a semidefinite program. 
The optimization above runs over the cone of all positive selfadjoint matrices, i.e. those matrices $\rho$ for which equivalently\cite{Blackadar2006}
\begin{itemize}
    \item[(i)] all eigenvalues are positive and $\rho$ is self-adjoint 
    \item[(ii)] all expectations $\Tr{\rho XX^*}$ are positive
    \item[(iii)]a decomposition $\rho=XX^*$ exists.
\end{itemize}
It should be noted that the cone of positive semidefinite (PSD) matrices  plays two roles in the case of conventional SDPs. On one hand, PSD matrices are  the positive elements of $\mathbb{C}^{n \times n}$ from a purely algebraic perspective. This notion of positivity only depends on the form of the element itself. This corresponds to the definition (iii) with $\rho = XX^{*}$ or to definition (i) by using the spectrum of $\rho$. On the other hand, PSD matrices are also used to \textit{label} the set of positive linear functionals, see  definition  (ii). Indeed, by Riesz representation it can be shown that any positive linear functional $\omega: \mathbb{C}^{n \times n} \rightarrow \mathbb{C}$ is of the from $\omega(X) = \Tr{\rho X}$ for a positive $\rho$. Consequently, each of those properties (i)-(iii) could in principle give rise to a variety of different generalizations. 
Initially the class of optimization problems that is considered in this work could be understood as a generalization of property (ii):  we generalize the matrices $\{F_0,F_1,\dots,F_m\}$ by self-adjoint elements of a $C^*$-algebra $\mathcal{A}$ and consider optimizations over the cone of functionals that are positive semidefinite on all elements of the form $XX^*$.



\subsection{Generalized SDPs}
With respect to the usual matrix product the vector space $\mathbb{C}^{n\times n}$ naturally admits the structure of an algebra, on which it is somewhat natural and well motivated by many applications to consider the hermitian conjugation
\begin{align}
    X^*:=X^\dagger=\overline{X}^T
\end{align}
as an involution and 
\begin{align}
   \Vert X\Vert := \sup \lbrace \, \Tr{\rho X}  \, \vert \,   \rho \geq 0 \und \Tr{\rho}=1 \rbrace
\end{align}
as a norm. In the following we will write $M_n(\mathbb{C}):=(\mathbb{C}^{n\times n},\cdot,\phantom{x}^*,\Vert\cdot \Vert)$ in order to denote $\mathbb{C}^{n \times n}$ equipped with this extra structure. Further this norm shares the property of being submultiplicative, that is for any $X,Y \in M_{n}(\mathbb{C})$ one has $\Vert XY \Vert \leq \Vert X \Vert \, \Vert Y \Vert$ and turns $M_{n}(\mathbb{C})$ into a complete space.

The abstraction of this structure, i.e. $M_{n}(\mathbb{C})$ is nowadays known as a $\cs$-algebra $\mc A$: a Banach, i.e. norm complete, algebra equipped with a $*$-involution that fulfills the so called $\cs$-property $\Vert XX^*\Vert=\Vert X\Vert^2$\cite{Blackadar2006}. The deceptively simple axioms turn out to be extremely powerful and forcing a rigid structure on a $\cs$-algebra. In particular, it is possible to retrieve the initial example of $M_{n}(\mathbb{C})$ from the abstract definition in case of finite dimensionality: any finite dimensional $\cs$-algebra is isomorphic to a direct sum of matrix algebras (Artin-Wedderburn Theorem). 

However, in a general algebra $\mathcal{A}$, the three definitions (i)-(iii) are not longer equivalent and we have to make a decision for defining what a generalized SDP is. Given a general algebra $\mathcal{A}$ it makes sense to call an element $F \in \mathcal{A}$ selfadjoint, if $F=F^*$ holds with respect to the $*$-involution. Furthermore, we get a generalization of positive semidefiniteness by generalizing property (iii), i.e. we say that an element $P\in\mc A$ is positive semidefinite, written $P\geq0$, if it admits a decomposition $P=XX^*$. Alternatively one could think about generalizing property (i) by demanding $P$ to have a non-negative spectrum. However, it turns out that this definition would coincide with $P=XX^*$, which is in fact a distinct property of $\cs$-algebras\footnote{Indeed, for the more general concept of a $*$-algebra, where one drops the requirement of a complete norm, there exist cases where $\rho = XX^*$ can have a negative eigenvalues.}.
The set of all those positive elements forms a closed convex cone in $\mathcal{A}$, which will be denoted by $\pp{A}$. The dual of this cone will be denoted by $\pps{A}$ and can be defined algebraically\cite{Blackadar2006} by
\begin{align}\label{eq:definition_dualcone}
    \pps{A} := \{\rho:\mc A \mapsto \mathbb{C} \ | \ \text{ linear},\quad \rho(XX^*)\geq 0 \quad\forall X\in\mc A\} .
\end{align}

Now we are able to state the generalized version of a conventional SDP with operations on an abstract $\cs$-algebra. Motivated by applications in which $\rho$ takes the role of a quantum state, we take the path of replacing $\mathcal{P}_+(\mathbb{C}^{n\times n})^\star$ by $\pps{\mathcal{A}}$ and therefore consider optimization problems of the following type: 
\begin{definition}\label{def:sdp}
For $m\in \mathbb{N}$, let $\{F_0,F_1,\dots F_m\}$ be a list of selfadjoint elements from  a $\cs$-algebra  $\mc A$ and let $\{f_1,\dots,f_m\}$ be a list of real numbers. The conic optimization problem 
\begin{align}
\label{sdpdef}
    &&\inf \quad&\omega( F_0) &&\quad\\
    &&\sth \quad&\omega(F_i)\leq f_i &&\forall i\in\{1,\dots,m\}\nonumber \\
    &&\quad \quad&\omega\in\pps{\mathcal{A}}&&\qquad
\end{align} is called the primal form of a semidefinite program on $\mc A$, short $\cs$-SDP. 
\end{definition}

Same as in the case of ordinary SDP, we can define feasible points.
\begin{definition}
    A $\cs$-SDP is called feasible, if there is an element $\omega \in \pps{ \mathcal{A}}$ with 
    \begin{align}
        &&\omega(F_i) \leq f_i && \forall i \in \{1,\dots,m\}.
    \end{align}
    Accordingly an optimal solution is an element $\tilde{\omega}$, such that for all feasible solutions $\omega \in \pps{\mathcal{A}}$
    \begin{equation}
        \omega(F_0)\geq \tilde{\omega}(F_0)
    \end{equation}
    holds true.
\end{definition}
Because well-known finite dimensional SDPs are included in this generalization, we must be appropriately careful when adding a value to the program. This is not important for the first statements, since we first specify bounds for the value. It only gets interesting when we define and use the dual program.
Therefore, the \textit{value} (possibly infinity) of a $\cs$-SDP is given by
\begin{equation}\label{eq:value}
    \tilde{c} := \inf \{\omega(F_0) \ | \ \omega(F_i) \geq f_i \quad \forall i \in \{1,\dots,m\}\}, \quad \omega \in \pps{\mathcal{A}}.
\end{equation}

Our original interest on this class of problems stems from the huge variety of well-known problem types that can obtained as particular instances of \eqref{sdpdef}. In the next section we present examples of how those problems can be rephrased. 

However, for understanding the theoretical insights the interested reader can skip the next section. 


\subsection{\texorpdfstring{Examples of $\cs$-SDPs}{}}\label{subsec:examples_sdp}
In the following we will rephrase two important optimization problems, namely the problem of polynomial optimization and the problem of optimizing over the set of quantum correlations as a $\cs$-SDP. 
\begin{example}[The Generalized Moment Problem]\label{ex:lasserre}
    We consider a the generalized moment problem (GMP), which can be solved Lasserre's famous polynomial optimization method \cite{lasserre}. Without sake of completeness we introduce concepts from Lasserre's hierarchy which are important for us later in this work. Consider a polynomial $f$ in $d$ variables and a compact set $K \subset \mathbb{R}^d$ and the optimization task
    \begin{equation}\label{eq:polynomial_optimization}
        \inf \{f(x) \ | \ x \in K\}.
    \end{equation}
    One can show\cite{lasserre} that \eqref{eq:polynomial_optimization} is equivalent to the following optimization task:
    \begin{align*}
        \inf \quad & \mu(f) \\
        \sth \quad & \mu(1) = 1  \\
        &\mu \in \mathcal{M}(K)_+
    \end{align*}
    whereby $\mathcal{M}(K)_+$ is the set of positive Borel measures; i.e. the topological dual space of $C(K)$. In general a GMP is given by
    \begin{align}\label{eq:GMP}
        \rho_{\text{mom}} = \sup_{\mu \in \mathcal{M}(K)_+} &\int_{K} f d\mu \\ \nonumber
        \text{s.th.} \ &\int_K h_j d\mu \leq h_j \quad \forall 1\leq j \leq m
    \end{align}
    for multivariate polynomials $f,h_1,\ldots,h_m$ and real numbers $\gamma_1,\ldots,\gamma_m$.
    
    one possible point of view would be the following. Consider the set of finite positive Borel measures and in particular the \emph{moment sequence} $y = (y_\alpha)_{\alpha \in \mathbb{N}^n}$ for each of them given by
    \begin{align*}
        y_\alpha = \int_K x^\alpha d\mu.
    \end{align*}
    If we introduce the functional
    \begin{align*}
        L_y: \mathbb{R}[x] &\to \mathbb{R} \\
        f &\mapsto L_y(f) := \sum_{\alpha \in \mathbb{N}^n} f_\alpha y_\alpha
    \end{align*}
    for $f = \sum_{\alpha} f_\alpha x^\alpha$, so we can reformulate the optimization problem \eqref{eq:GMP} to 
    \begin{align*}
        \sup_{y \ \text{moment sequence}} &L_y(f) \\
    \text{s.th.} \ &L_y(h_j)\leq \gamma_j \quad \forall 1\leq j \leq m.
    \end{align*}
    The question whether $y$ is moment sequence can be answered with linear constraints of the following type. Define for a multivariate polynomial in $n$ variables \emph{moment matrices} to be
    \begin{align*}
        M_r(u y)_{\alpha,\beta} :=  (L_y(u x^\alpha x^\beta))_{\alpha,\beta}
    \end{align*}
    whereby we cut the degree up to $r$. Then Theorem $3.8$\cite{Lasserre_book} states that $y$ has a finite Borel representing measure with support contained in a semi-algebraic set $K = \{x \in \mathbb{R}^n \ \vert \ g_j(x)\geq 0 \quad \forall 1\leq j \leq m$\footnote{There is in addition a compactness criteria for technical reasons.} if and only if
    \begin{align*}
        L_y(f^2g_J) \geq 0 \quad \forall J \subset \{1,\ldots,m\} \quad \forall f \in \mathbb{R}[x]
    \end{align*}
    or equivalently 
    \begin{align*}
        M_r(g_Jy)\geq 0 \quad \forall J \subset \{1,\ldots,m\} \quad \forall r \in \mathbb{N}.
    \end{align*}
    One can think about these constraints as semidefinite constraints for $y$ being a moment sequence. In particular we can introduce the dual program
    \begin{align}\label{eq:dual_GMP}
        \rho_{\text{pop}} = \inf_\lambda \sum_{j=1}^m &\gamma_j \lambda_j \\
        \text{s.th.} \ \sum_{j=1}^m &\lambda_j h_j(x) - f(x)\geq 0 \quad \forall x \in K \\
        \lambda_j \geq 0 &\quad \forall 1\leq j \leq m.
    \end{align}
    I.e. the problem is equivalent in checking whether a polynomial is positive over $K$.
    For basic semi-algebraic sets there is Putinar's Positivstellensatz\cite{Putinar} which tells us how positive polynomials look like on these type of sets and therefore one can basically relax the dual problem \eqref{eq:dual_GMP} to a finite level version and solve the corresponding semidefinite program. 

    The interesting observation of this story is that from a global point of view \eqref{eq:GMP} is something what we will later call $\cs$-SDP, because it is an optimization problem over the dual space of the $\cs$-algebra $C(K)$ for the compact set $K$. Apparently this does not solve the description of the dual space of $C(K)$, i.e. due to Riesz famous representation theorem $\mathcal{M}(K)$, but in general it leads the eyes to the question whether one can decide positivity in a $\cs$-algebra. And this is theoretically easily done: in a $\cs$-algebra $\mathcal{A}$ an element $x \in \mathcal{A}$ is positive if and only if there is a $\star$-root, i.e. there is $y \in \mathcal{A}$ such that $x = y^\star y$. That is, there is somehow the difference of being positive in terms of polynomials, which is apparently on the one hand a hard question and solved by Putinar's Positivstellensatz and on the other hand it seems to be quite easy because we have roots at our disposal. 
\end{example}

\begin{example}[NPA]\label{ex:NPA}
    Essentially the NPA hierarchy\cite{navascues2007} examines the question whether there is a quantum system for a given joint probability distribution. To be concrete, consider a bipartite experiment and a bipartite probability distribution. The question is whether the bipartite probability distribution can be generated by local measurements of two parties. In a mathematical sense this ansatz works with free $\star$-algebras. It chooses operators and works then with representations on Hilbert spaces of this free $\star$-algebras. 
    
    Comparing to this ansatz we assume in the paper presented here always the existence of a $\mathcal{C}^\star$-algebra at the beginning. This clarifies mathematically a lot, because we have a direct notion of positivity which is a priori not clear in NPA (one has to construct positivity from representations). Moreover, we have a simple notion of 'quadratic module', it will be a set of squares. However, the relaxed semidefinite programs in the end can become the same in particular cases. 
\end{example}

In conclusion all the selected examples among many others presented above have the common structure of a $\mathcal{C}^\star$-algebra and an optimization over a feasible region in the dual cone. Therefore it is worth to construct a theory how to solve in general such problems assuming only the global structure.









\def\jnc{\operatorname{JNC}}%

\section{Finite dimensional relaxations of generalized SDPs}\label{sec:finite_dim_relax}

\subsection{The joint numerical range}\label{subsec:joint_numerical}
In contrast to conventional SDPs, which are  considered to be efficiently solvable, finding the solution to a generalized SDP may turn out to be a challenging task. 

In the following we will introduce methods for approximating a generalized $\cs$-SDP by a finite dimensional one. 
The leading intuition for the existence of such a finite approximation to a $\cs$-SDP \eqref{sdpdef} stems from the observation that, even though we may consider a wild and potentially infinite dimensional algebra $\mc A$, all the relevant information for an optimization like \eqref{sdpdef} is in principle encoded in the finite (indeed not more that $N$ dimensional)  subspace $\mc F\subset \mc A$ that is spanned by  $\vec F :=\{F_0,\dots, F_N\}$ and in an appropriate subset of its dual $\mathcal{F}^\star$. 
It is easy to see that  we can, at least formally, replace the optimization in  \eqref{sdpdef} by only considering functionals from the convex cone
\begin{align}
 \mc F_{+}^{\star}:=\pps{\mc A}\vert_{ \mc F} \subset \mc F^\star.
\end{align} 
Indeed, suppose we have a functional $\omega \in \pps{\mathcal{A}}$ that is a solution of \eqref{sdpdef}. By definition, $\omega$ fulfills all constraints and thus also $\omega \vert_{\mathcal{F}}$ does. Clearly, we also have $\omega(F_{0}) = \omega \vert_{\mathcal{F}} (F_{0})$, hence the optimal value is independent of whether the restricted version $\omega \vert_{\mc F}$ or the full supported functional $\omega$ is considered. 
There is also a more direct way to see the finite dimensional object standing behind an generalized SDP. Consider the set
\begin{align}\label{jnc}
\jnc_{\vec F }: =\{\vec y_\omega=(\omega(F_0),\dots,\omega(F_N))|\omega\in\pps{\mc A}\}\subset\mathbb R^N
\end{align}
to which we will refer as joint numerical cone (JNC). This name is chosen in correspondence to the well-known concept of a (convex) joint numerical range, which is defined similarly to \eqref{jnc}.\footnote{ The only difference is that in the definition of the (convex) joint numerical range $\omega$ runs over all normalized states of a unital $\cs$-algebra. Moreover, the joint numerical range can be seen as a cone base for the JNC.}
$\jnc_{\vec F}$ can, indeed, be understood as embedding of $\mathcal{F}^\star_+$ into $\mathbb R^N$, in the sense that for any $f\in\mc F$ parameterized by coefficients $\vec\lambda$ with $f=\sum\lambda_\mu F_\mu$, we have 
\begin{align}
    \omega(f)=\sum \lambda_\mu \omega(F_\mu) = \vec \lambda.\vec y_{\omega}
\end{align}
Accordingly we can state any general SDP as explicit conic program 
\begin{align}
\label{sdpjnc}
    &&\inf \quad&\vec y .\mathbf e_0 &&\quad\\
    &&\sth \quad&\vec y \leq \vec f &&\quad\nonumber \\
    &&\quad \quad&\vec y\in\jnc_{\vec F }&&\qquad
\end{align}
over the finite dimensional cone $ \jnc_{\vec F }$. 

Despite that this view is giving some nice perspective on the problem, it is not yet doing any work. 
Here, the  caveat is that the exact shape of $\mc F_+^\star$ or  $ \jnc_{\vec F }$ is generally inaccessible, since it still contains global information about the cone $\pps{\mc A}$. It is worth to stress that the shape of the cone $\mathcal{F}_+^\star$ has a priori no common properties with the usual positive cone in a finite dimensional vector space. It inherits a term of positivity from the global structure. Roughly our main task is to track positivity through the whole approximation. The mechanism of  the following approximation technique can be therefore understood as a construction that provides us with numerically accessible  outer approximations to $\mc F_+^\star$ or equivalently  $ \jnc_{\vec F }$. 

\subsection{Construction of a Relaxation}
In this section we discuss how to approximate the value of a $\cs$-SDP with theoretical methods and a finite dimensional relaxation. Recall from \eqref{eq:definition_dualcone} that the cone $\pps{\mathcal{A}}$ is the dual cone of $\pp{\mathcal{A}}$. It means that inner approximations of the cone $\pp{\mathcal{A}}$ lead to outer approximations of the cone $\pps{\mathcal{A}}$. But outer approximations of the dual cone are giving lower bounds to a $\cs$-SDP. 

In order to connect to the two very nice points of view in \autoref{sec:prelim_exam} and \autoref{subsec:joint_numerical} we have to attenuate the ideas from \autoref{sec:prelim_exam} and abstract the ideas from \autoref{subsec:joint_numerical}. In \autoref{sec:prelim_exam} we enforce the map $T$ to be a quantum channel, i.e. completely positive and trace-preserving. As it turns out firstly it is very difficult to find the right channel which dual maps a known family of operator $\{M_0,M_1,\dots,M_m\}$ to operators $\{F_0,F_1,\ldots,F_m\}$ for a $\cs$-SDP. Secondly it is not clear how to appropriate describe the set of states after applying the channel. 

To circumvent the first problem we use the dual point of view and construct a linear mapping $\phi$ which is not completely positive and trace-preserving anymore from a finite dimensional space into the algebra $\mathcal{A}$ of interest. Since positivity of a map implies positivity of the dual, we can make similar arguments even without completely positivity. 

In a first step we can choose random positive elements from the $\cs$-algebra to approximate the positive cone $\mathcal{P}_+(\mathcal{A})$. But then we need an assignment between these abstract elements and a cone in $\mathbb{C}^{n\times n}$ for an appropriate size $n$. The important point is that we need a way to translate positivity to a finite calculable system. The following fundamental observation, which is in fact a generalization of the idea of 'sum of squares'\cite{lasserre} in polynomial optimization or in the same manner from the $\star$-isomorphism in symmetry reductions\cite{de2007reduction} in generic SDPs gives a way for translating positivity to a controllable cone. 

\begin{lemma}[Fundamental Lemma] \label{lem:fundamental_lem}
    Let $\mathcal{A}$ be a $\cs$-algebra and consider a set 
    \begin{align}
    \gamma=\{\gamma_1,\dots,\gamma_n\} \subset \mathcal{A}.
    \end{align}
    Then the mapping 
    \begin{align}
       \phi: \mathbb{C}^{n\times n} &\to \mathcal{A} \\
        M &\mapsto \sum_{i,j = 1}^n m_{ij} \gamma_{i}\gamma_j^\star
    \end{align}
    is linear and positive. In particular let $x = yy^\star \in \pp{\mathcal{A}}$ and $y \in \operatorname{span}(\gamma)$. Then there exists $M \geq 0$ such that $\phi(M) = x$.  
\end{lemma}
\begin{proof}
    Linearity is clear. We have to show positivity. \\
    Positive elements are hermitian in $\mathbb{C}^{n\times n}$ and therefore each positive elements have a spectral decomposition. Therefore let $M \in \mathbb{C}^{n\times n}$ be positive. Then 
    \begin{equation*}
        M = \sum_{k} p_k P_k
    \end{equation*}
    with $p_k \in \mathbb{R}_{>0}$ and $P_k = c^{(k)}(c^{(k)})^\star$ one dimensional projections. Therefore 
    \begin{align}
        \phi(M) &= \sum_{k} p_k \phi(c^{(k)}(c^{(k)})^\star) \\
        &= \sum_k p_k \sum_{i,j} c^k_i\overline{c^k}_j \gamma_i \gamma_j^\star \\
        &= \sum_k p_k (\sum_i c^k_i \gamma_i) (\sum_j c^k_j \gamma_j)^\star.
    \end{align}
    But then we see that each sumand over $k$ is a square in $\mathbb{A}$ and therefore positive. Positive sums of positive elements are positive which concludes the proof. 
\end{proof} 

The \autoref{lem:fundamental_lem} shows us that for a given sequence of elements $\gamma \subset \mathcal{A}$ there is a way of translating structure from a $\cs$-algebra to a finite dimensional space. But in general to use the preimage of $\phi$ for an approximation needs the structure of the kernel of $\phi$ because the kernel contains the important inner structure of the algebra $\mathcal{A}$. In other words the kernel decides whether two parameterized elements $M,M^\prime \in C^{n\times n}$ are equal to each other. To get a working approximation we have to have in mind that this is precisely the point where we have to involve the knowledge about the algebra. For example the relations of generators of the algebra. 

We are now in the situation to present a working approximation. Recall that, in fact, we want to approximate $\mathcal{F}_+^\star$ from \autoref{subsec:joint_numerical}. For this goal we optimally need to build a hierarchy\footnote{actually always intersected with $\mathcal{F}$} of numerically accessible cones
\begin{equation*}
    \Sigma_2^{(1)} \subset \ldots \subset \Sigma_2^{(n)} \subset \Sigma_2^{(n+1)} \subset \ldots \subset \pp{\mathcal{A}}.
\end{equation*}
These cones should be the images of positive cones in $\mathbb{C}^{n\times n}$ in \autoref{lem:fundamental_lem} for different sizes of $n$. 
The sequence of dual cones, as defined in \eqref{eq:definition_dualcone} leads then to a sequence
\begin{equation}\label{eq:hierarchy}
    (\Sigma_2^{(1)})^\star \supset \ldots \supset (\Sigma_2^{(n)})^\star \supset (\Sigma_2^{(n+1)})^\star \supset \ldots \supset \pps{\mathcal{A}}.
\end{equation}
These are outer approximations of $\pps{\mathcal{A}}$. 
To construct such sequences of cones we start with a sequence of elements in the $\cs$-algebra $\mathcal{A}$
\begin{align}
    \gamma=\{\gamma_1,\dots,\gamma_n\}.
\end{align}
These elements $\gamma$ span a subspace $V_n \subset\mc A$, i.e. 
\begin{align}
    V_n := \operatorname{span}(\gamma)=\left\{v\in\mc A\middle|v=\sum_{i=1}^n v_i \gamma_i, \text{ for some } v_1,\dots, v_n \in \mathbb{C}\right\},
\end{align}
from which we can construct a space\footnote{the notation $\operatorname{conv}(\Omega)$ denotes the convex hull of a set $\omega$} 
\begin{align}
    Q_n := \operatorname{conv}\{vw^\star \ | \ v,w\in V\}=\operatorname{span}\{\gamma_i\gamma_j^\star\} = \operatorname{Im}(\phi)
\end{align}
that contains all sums of the products that could be build from elements from $V_n$. To ensure that we get senseful statements about $\mathcal{F}$ we have to consider sequences $\gamma$ that are constructed  such that the subspace $\mc F$ is  contained in $Q_n$, because then the cone $\mathcal{F}_+^\star = \pps{\mathcal{A}} \cap \mathcal{F}^\star$ is contained in \eqref{eq:hierarchy}. Recall that this is exactly the difficulty with our channel in \autoref{sec:prelim_exam}. Starting from the functionals, in general we can not control which element is the right preimage under the dual mapping $T^\star$ there.  
In unital algebras this requirement could always be fulfilled by appending $\{F_0,\dots, F_N, \mathbb{I} \}$ to an initial sequence $\gamma_0$. An alternative, which also works on non-unital algebras, would be to append the square roots  $\sqrt{F_i}$. 

Then we consider the cone 
\begin{align}
    \Sigma_2^{(n)} =\operatorname{conv}\{vv^\star \ | \ v\in V \} \label{sigma2}
\end{align}
of 'sums of squares' in $Q_n$. Indeed, as $\text{dim} (Q) \leq n^{2}$ and $\Sigma_{2} \subset Q$, we also have $\text{dim}(\Sigma_{2}) \leq n^{2} < \infty$. 

Since in a $\cs$ algebra, all squares, as well as their convex hull, are positive we obtain
\begin{align}
 \Sigma_2^{(n)} \subseteq\mc P^+(\mc A), 
\end{align}
which implies \eqref{eq:hierarchy} for the corresponding dual cones. This however implies that the finite cone $\Sigma_2 ^\star |_{\mc F}$ (that no longer contains inaccessible global information) is a super set 
\begin{align}
    (\Sigma_2 ^{(n)})^\star |_{\mc F}\supseteq\mc F_+^\star. 
\end{align}
Hence we can conclude
\begin{prop}
The value of a $\cs$-SDP \eqref{sdpdef}, defined by the sets $\{F_0,\dots, F_m\}$ and  $\{f_1,\dots,f_m\}$, is bounded from below by the optimization
\begin{align}
\label{sdprelax}
    &&\inf \quad&\omega(F_0) &&\quad\\
    &&\sth \quad&\omega(F_i)\leq f_i &&\forall i\in\{1,\dots,m\}\nonumber\\
    &&\quad&\omega\in (\Sigma_2^{(n)}) ^\star \vert_{\mc F} &&\quad\nonumber.
\end{align}

\end{prop}

We want to mention that there are somehow some open questions from a practical point of view, which we tackle in the next section. In particular, we did not talk about any properties of the sequence $\gamma$ and their properties of approximating the positive cone of the $\cs$-algebra.




\subsection{The approximation as a conventional SDP}
For this section we choose a fixed set $\gamma$ and corresponding spaces $V,Q$ and cone $\Sigma_2$. 

The optimization \eqref{sdprelax} is in fact  not only finite from a formal perspective rather than also from  actual practicality since it can be explicitly formulated by a matrix valued SDP of $n\times n$ matrices, as we will show next.  

Recall the mapping $\phi$ from \autoref{lem:fundamental_lem}, defined by
\begin{align}
    \phi: \mathbb{C}^{n\times n}\rightarrow Q \quad \phi( |i\rangle\langle j|)=\gamma_i\gamma_j^\star,
\end{align}
i.e. the map that maps a matrix $M=\sum m_{ij}|i\rangle\langle j|$ to an algebra element 
\begin{align}
    \phi(M)=\sum m_{ij} \gamma_i\gamma_j^\star. \label{defphi}
\end{align}
This map is surjective, since \eqref{defphi} merely describes the span of $\gamma_i\gamma_j^\star$, i.e., $Q$. 
Let $\operatorname{ker}(\phi) \subset \mathbb{C}^{n \times n}$ be the kernel of $\phi$ and denote by $\pi$ the projection map $\mathbb{C}^{n \times n} \rightarrow \mathbb{C}^{n \times n} / \operatorname{ker}(\phi)$, which assigns to each element its coset. We write $[x]:=\{x+k| k\in \operatorname{ker}(\phi)\} = \pi^{-1} (\phi(x))$ denote  the equivalence class induces by $\operatorname{ker}(\phi)$. As $\phi$ is surjective, we obtain from isomorphism theorem, that the map $\hat{\phi}$ is an isomorphism. 
\begin{figure}[!h]
\begin{center}
\begin{tikzpicture}
  \matrix (m)
    [
      matrix of math nodes,
      row sep    = 3em,
      column sep = 4em
    ]
    {
      \mathbb{C}^{n \times n}              &  Q \\
       \hat{Q} &             \\
    };
  \path
    (m-1-1) edge [->] node [left] {$\pi$} (m-2-1)
    (m-1-1.east |- m-1-2)
      edge [->] node [above] {$\phi$} (m-1-2)
    (m-2-1.east) edge [->,dashed] node [below] {$\exists! \ \hat{\phi}$} (m-1-2);
\end{tikzpicture}
\end{center}
\caption{We define $\hat{Q} := \mathbb{C}^{n\times n}/ \operatorname{ker}(\phi)$. The figure tells us that we get an isomorphism between $Q$ and $\hat{Q}$ from the isomorphism theorem. In other words this means that we constructed an image of $Q$ with the right relations between elements in $Q$ in a accessible space $\mathbb{C}^{n\times n}$.}
\label{fig_iso}
\end{figure}

Now we can state the main proposition of this section.
\begin{prop}[Identification of $Q^\star$]\label{thm:identification_Qstar}
    An element $g \in (\mathbb{C}^{n\times n})^\star$ can be identified up to equivalence with a linear functional in $Q^\star$ if and only if $g$ annihilates $\operatorname{ker}(\phi)$ or in other words
\begin{equation*}
    g(M) = 0 \quad \forall M \in \ker(\phi).
\end{equation*}
\end{prop}
\begin{proof}
The map in \autoref{fig_iso} shows us that $\hat\phi: \hat Q \rightarrow Q $ that acts on  $\hat Q $ via 
\begin{align}
    \hat \phi(x) := \phi([x])
\end{align}
is an isomorphism. To describe $\hat{\Sigma}_2^\star$ we need to include the isomorphism $\phi$ in an appropriate manner in our discussion. For this purpose we consider the dual mapping
\begin{align}
    \hat{\phi}^\star: Q^\star &\to \hat{Q}^\star \\
    q^\prime &\mapsto q^\prime \circ \phi.
\end{align}
Basic linear algebra states that $\hat{\phi}^\star$ is an isomorphism. Furthermore a well-known result from functional analysis states that
\begin{align}
    \pi^\star: (\mathbb{C}^{n\times n}/\operatorname{ker}(\phi))^\star &\to \operatorname{\ker}(\phi)^\perp \\
    \hat{q}^\prime &\mapsto \hat{q}^\prime \circ \pi
\end{align}
 is an isomorphism as well. Therefore, we have the chain of isomorphisms
 \begin{align}
        Q^\star \stackrel{\hat{\phi}^\star}{\longrightarrow} \hat{Q}^\star \stackrel{\pi^\star}{\longrightarrow}  \operatorname{\ker}(\phi)^\perp,
 \end{align}
 whereby the annihilator is defined via
 \begin{align}
     \operatorname{ker}(\phi)^\perp := \{g \in (\mathbb{C}^{n \times n})^\star \ | \ g(M) = 0 \quad \forall M \in \operatorname{ker}(\phi)\}.
 \end{align}
\end{proof}
The proof of \autoref{thm:identification_Qstar} shows that quotient spaces become subspaces in duality. This is an important observation throughout the whole discussion. 

This simple but very powerful \autoref{thm:identification_Qstar} leads us to an abstract but nice description of elements in $\hat{Q}^\star$ and therefore of the moment matrices as generalizations of polynomial optimization in \autoref{ex:lasserre}:
\begin{cor}
    The relation between linear functionals from $Q^\star$ and $\hat{Q}^\star$ is given by
    \begin{equation*}
        \tr(\Gamma^\sigma M) = \sigma(\phi(M)).
    \end{equation*}
    For $\Gamma^\sigma := \phi^\star(\sigma)$. These matrices are also called \textit{moment matrices} in other hierarchies. 
\end{cor}
\begin{proof}
    This is clear from the definition of the dual mapping. 
\end{proof}

Including subspaces in a SDP is therefore simple, because constraints of the manner 
\begin{equation*}
    \tr(\rho K) = 0
\end{equation*}
correspond to functionals $\tr(\rho \cdot)$ which have the subspaces spanned by $K$ in the kernel. Therefore, if we add all generators of the kernel of $\phi$, we create the annihilator.  

For our purpose there is only left an open discussion about the kernel of $\phi$ and how to include such a kernel. But as we will see, it is in some sense the task of a physicist to know his experiment in such a manner that he can decide whether two experimental setups, i.e. mathematical models are equal to each other. Our examples will demonstrate this impressively. 
In conclusion of the discussion we can state the following main theorem of our paper:

\begin{theorem}[Relaxation]\label{thm:relaxation}
The relaxation \eqref{sdprelax} of a $\cs$-SDP that is build from a finite sequence $\{\gamma_1, \dots, \gamma_n\}$ can be formulated as an $n\times n$ matrix valued SDP, and therefore efficiently computed.  A concrete form of this SDP is obtained from any set of  matrices $M_{F_i}$ with
\begin{align}
    \hat \phi ([M_{F_i}]) =F_i
\end{align} 
and any set of $m_k$ matrices $\{K_i\}$ with 
\begin{align}
    \operatorname{span}\{K_j\}= \operatorname{ker}(\phi)
\end{align} and given by 
\begin{align}
\label{sdprelaxmat}
    &&\inf \quad& \tr( \rho M_{F_0}) &&\quad\\
    &&\sth \quad& \tr( \rho M_{F_i})\geq f_i &&\forall i\in\{1,\dots,m\}\nonumber\\
    &&\qquad \quad& \tr( \rho K_j) = 0 && \forall K_j\in\{1,\dots,m_k\}\nonumber\\
    &&&\rho\geq 0&&\nonumber
\end{align}
\end{theorem}

\subsection{CHSH worked out}
We close this section with a familiar example.
\begin{example}[The algebra generated by two projections]
    One of the most famous examples of abstract language in quantum theory is the algebra generated by two projections. Consider a set of generators
    \begin{align*}
        \mathcal{G} := \{1,A_0,A_1,B_0,B_1\}
    \end{align*}
    which should fulfill the following relations
    \begin{align*}
        \mathcal{R} := \{ A_i^2 = 1, B_i^2 = 1, B_i^\star = B_i, A_i^\star = A_i, [A_i,B_j] = 0 \ \vert \ 1\leq i, j \leq 2\}.
    \end{align*}
    It is well known that the algebra spanned by the generators and relations $C(\mathcal{G} \ \vert \ \mathcal{R})$ becomes a $\cs$-algebra\cite{halmos}. 
    Considering now a Bell inequality\cite{Bell:111654} 
    \begin{align*}
       F_0 := A_0B_0 + A_1B_1 + A_0B_1 - A_1B_0
    \end{align*}
    leads directly to a $\cs$-SDP in the following sense:
    \begin{align}
    &&\inf \quad&\omega( F_0) &&\quad\\
    &&\sth \quad&\omega(1)= 1&& \nonumber \\
    &&\quad \quad&\omega\in\pps{\mathcal{A}}&&\qquad.
\end{align}
The interpretation of this problem is to optimize over all states for finding the maximal violation of the Bell inequality above (CHSH type).\\
Solving this example with our method needs a sequence $\gamma$. One can choose for example in a first step
\begin{align*}
    \gamma = (1,A_0,A_1,B_0,B_1,A_0B_1,A_0A_1,B_0B_1).
\end{align*}
Furthermore the relations $\mathcal{R}$ become directly elements in the kernel of $\phi$ in the discussion above. Therefore working out the kernel is not difficult in this example for short sequences $\gamma$. In conclusion one can directly verify with the rules of \autoref{thm:relaxation} that the optimal value $\sqrt{2}2$ will be achieved in a nummerical example. 
\end{example}

\section{Convergence of the hierarchy}\label{sec:convergence}

In this section we aim to clearify under which circumstances we can prove convergence of the hierarchy by iterating \autoref{thm:relaxation}. The question now is how the optimal value of the relaxation of the $\cs$-SDP relates to its true value. 

For this we need the notion of base property. 

\begin{definition}[Base Property]
    Let $(\gamma^{(n)})_{n \geq 1}$ be a sequence in $\mathcal{A}$. Now define
    \begin{equation*}
        V_{n} := \text{span} ( \{\gamma^{(1)},\ldots,\gamma^{(n)} \}).
    \end{equation*}    
    We say that $(\gamma^{(n)})_{n \geq 1}$ has the base property, if for any $\epsilon >0$ and any $a \in \mathcal{A}$ there exists $n_{0} \in \mathbb{N}$ such that for all $n \geq n_{0}$ there is $b \in V_{n}$ with $\vert \vert a - b \vert \vert < \epsilon$.
\end{definition}

We emphasize at this point that there is an obvious way to translate the base property into separability of the $\cs$-algebra. Namely choose a sequence with base property and consider the vector space $V_n$ spanned by the subfield $\mathbb{Q}+ i\mathbb{Q}$ of $\mathbb{C}$ (this is in fact the field extension $\mathbb{Q}(i)$). 

For example all $\cs$-algebras with finite set of generators are automatically separable and fulfill therefore the base property.  

\begin{theorem}
We consider a $\cs$-SDP
\begin{align*}
    c^\star = &\inf  \rho(F_0) \\
    &\text{s.th.} \ \rho(F_i)\leq f_i 
\end{align*}
and assume that we have base property, strong duality at all levels, i.e. in particular
\begin{align*}
    c^\star = &\sup_{\lambda} f^T \lambda \\
    &\text{s.th.} \ \sum_{i} \lambda_i F_i -F_0 \in \pp{\mathcal{A}}.
\end{align*}
Then we have $\vert c^{\star} - c^{n} \vert \xrightarrow{n \rightarrow \infty} 0$
with value $c^n$ for each level.  
\end{theorem}
\begin{proof}
Suppose that $x \in \mathcal{P}^{+}(\mathcal{A})$. By the $\cs$-property from \autoref{subsection:conventional_sdp}, there exists $z \in \mathcal{A}$ such that $x = z^{*} z$. By the base property, we can conclude that for sufficient large $n \in \mathbb{N}$ there exists $y \in V_{n}$ such that $\vert \vert y - z \vert \vert < \epsilon$. In particular we have that there is $d \in \mathcal{A}$ with $\vert \vert d \vert \vert < \epsilon$ such that $y+d =z$. Further $y^{\star} y \in \Sigma^{(n)}_{2}$. Therefore we have
\begin{align}
    \vert \vert y^{\star} y - z^{\star} z \vert \vert = \vert \vert (z-d)^{\star} (z-d) - z^{\star} z \vert \vert = \vert \vert d^{\star} z - z^{\star} d + d^{\star} d \vert \vert \leq \vert \vert d^{\star} z \vert \vert + \vert \vert z^{\star} d \vert \vert + \vert \vert d^{\star} d \vert \vert. 
\end{align}
By the submultiplicativity of the norm it follows that the left hand side can be done sufficiently small by choosing $\varepsilon>0$ in dependence of $z$ small enough. This means that we have 
\begin{align*}
     \Sigma_2^{(n)} \subset \Sigma_2^{(n+1)} \subset \ldots \subset \pp{\mathcal{A}} \quad n\to \infty
\end{align*}
as a limit of sets. Consider for
\begin{align*}
    L: \mathbb{R}^m &\to \mathcal{A} \\
    \lambda &\mapsto  \sum_{i} \lambda_i F_i -F_0
\end{align*}
the sets
\begin{align*}
    \Omega^{(n)} := L^{-1}(\Sigma_2^{(n)} \cap \text{im}(L))
\end{align*}
and in particular
\begin{align*}
    \Omega := L^{-1}(\mathcal{P}(\mathcal{A}) \cap \text{im}(L)),
\end{align*}
the set of all $\lambda \in \mathbb{R}^n$ such that $L(\lambda)$ is positive in $\mathcal{A}$. Then we have in particular
\begin{align*}
    \Sigma_2^{(n)} \cap \text{im}(L) \subset \ldots \subset \mathcal{P}(\mathcal{A}) \cap \text{im}(L) \quad n\to \infty.
\end{align*}
By strong duality on all levels we conclude
\begin{align*}
   c^{n} := \sup_{\omega \in \Omega^{(n)}} f^T\omega
\end{align*}
and 
\begin{align}\label{eq:proof_convergence}
    c^\star := \sup_{\omega \in \Omega} f^T\omega.
\end{align}
The sequence $(c^n)$ is monotone and bounded and therefore convergent towards $\sup_n c^n$. Assume for contradiction that there exists $\delta>0$ such that
\begin{align*}
    \vert \sup_n c^n  - c^\star \vert \geq \delta.
\end{align*}

Then we have in particular that for all $n \in \mathbb{N}$
\begin{align*}
    \vert \sup_{\omega  \in \Omega^{(n)}} f^T \omega  - c^\star \vert \geq \delta
\end{align*}
and therefore 
\begin{align*}
    \vert \sup_{\omega  \in \bigcup_n \Omega^{(n)}} f^T \omega  - c^\star \vert \geq \delta
\end{align*}
a contradiction to \eqref{eq:proof_convergence}. 
\end{proof}

This result has some metaphysical description. What we have done so far is build up on the assumption of having a $\cs$-algebra at hand. This means, that there is a notion of convergence inherently coming from the property of being a complete normed space. If we compare this with e.g. the NPA-hierarchy or even Lasserre's great polynomial optimization work, there is from a first view not clear what one means with 'convergence' in a topological sense. 

\section{\texorpdfstring{Symmetry reductions of SDPs as a special case of $\cs$-SDP}{}}\label{sec:symmetry_red}

Let us repeat the framework for symmetry reduction in SDPs from a theoretical perspective for finite groups. For this purpose consider a general SDP over a finite dimensional Hilbert space $\mathcal{H}$ of the following form
\begin{align}\label{eq:general_sdp}
    \max \ &\text{tr}[\rho F_0] \\
    \text{s.th.} \quad &\text{tr}[\rho F_j] \leq f_j \quad 1 \leq j \leq m \\
    &\rho \succeq 0.
\end{align}

A (finite) symmetry of this SDP is now a finite group $G$ with a representation\footnote{i.e. a group homomorphism}
\begin{align*}
    \Phi: G \to \text{GL}(\mathcal{H})
\end{align*}
such that 
\begin{align*}
    \Phi(g) F_j \Phi(g)^\dagger &= F_j \quad 1\leq j \leq m \\
    \Phi(g) F_0 \Phi(g)^\dagger &= F_0. \\
\end{align*}
We assume further that the representation is unitary, which means that $\Phi(g)^{-1} = \Phi(g)^\dagger$ for all $g\in G$. This assumption is for finite groups without loss of generality. The cyclicity of the trace now yields that a SDP which contains a finite symmetry has a solution which admits this symmetry. 

If we now assume access to a Hilbert-Schmidt orthonormal basis of the invariant subspace $\text{End}_{\mathbb{C}G}(\mathcal{H}) \subset \text{End}_\mathbb{C}(\mathcal{H})$
\begin{align}\label{eq:basis}
   \gamma := (B_1,\ldots,B_m) \subset \mathcal{B}(\mathcal{H})
\end{align}
we can build up our theory in the same manner as e.g. in \cite{de2007reduction}. Define for  
\begin{align*}
    B_i B_j = \sum_{k} \lambda_{ijk}B_k.
\end{align*}
the matrices $L_k \in \mathbb{C}^{m \times m}$ with the rule
\begin{align*}
    (L_k)_{i,j} := \lambda_{ijk}.
\end{align*}
Applying \autoref{lem:fundamental_lem} yields then the mapping
\begin{align*}
    \phi:\mathbb{C}^{m\times m} &\to \text{End}_{\mathbb{C}G}(\mathcal{H}) \subset \text{End}_{\mathbb{C}}(\mathcal{H})  \\
    L_k &\mapsto \sum_{i,j}\lambda_{ijk} B_iB_j^\dagger.
\end{align*}

Therefore it is easy to show that the optimal solution $\rho_{\text{opt}}$ is contained in $\text{End}_{\mathbb{C}G}(\mathcal{H})$, which is the set of all linear operators which commute with the operators $\Phi(g)$, i.e. the set of intertwiner and fulfills Slater's condition if the original SDP fulfilled Slater's 
condition. This observation states that the effective dimension of the SDP is in fact the $\mathbb{C}$-vector space dimension
\begin{align*}
    m := \text{dim}_{\mathbb{C}}\text{End}_{\mathbb{C}G}(\mathcal{H}).
\end{align*}
Therefore, in this case our hierarchy basically has been finished in one step, because the optimal solution is accessible in the dual space of $\text{End}_{\mathbb{C}G}(\mathcal{H})$. Furthermore there is no kernel to work out. The difficulty is really in defining the basis $(B_1,\ldots,B_m)$. For this task, however, there are methods as presented in de-Klerk et al.\cite{de2007reduction}.

\section{Conclusion}
In this paper we have taken a new look at known problems of optimization in quantum information theory. Our relaxation for $\cs$-SDPs is based on the fundamental observation that we are in fact only interested in the values of all positive functionals of a $\cs$-algebra on a finite vector space. Thus, all functionals of interest are characterized by a convex set in the dual space of $\mathbb{C}^{n\times n}$. This set can now be relaxed. We have chosen the way via the sum of squares in the style of Lasserre's hierarchy. Remarkably it, however, follows directly from our discussion how we can incorporate additional knowledge about, say, positive elements: one could either add them to $\Sigma_2^{(n)}$ directly or write them as SDP constraints. In contrast to Lasserre's hierarchy, we can not directly talk about local criteria about e.g. the dual space of a particular $\cs$-algebra $C(K)$ for a compact set $K$. But our general formulation gives hope that one may find in future local criteria for the $\cs$-algebra. In general one can apparently include the ideas of Lasserre, but then one has to choose the sequence $\gamma$, i.e. the basis of the vector space $Q$ in this case, to be polynomials. Then there is no way to present elements other than polynomials exactly. As this work pointed out, the idea of moment matrices can be generalized and is extremely powerful. We hope, that this general point of view will stimulate a new discussion about possible relaxations for general $\cs$-SDP in particular in quantum information theory. 
There is new work done by Klep et al.\cite{klep2023state}, which is able to tackle nonlinear constraints in a $\cs$-SDP similar notion. This has particular importance in causal structure problem as e.g. in Ligthart et al. \cite{Ligthart_2023, Ligthart_2023_1}
 
\section{Acknowledgements}
We thank Mario Berta, Tobias J. Osborne and Julius Zeiss for helpful discussions. We thank David Gross and Laurens Ligthart for organizing the workshop on SDP-hierarchies in May 2023 in Cologne. 
RS acknowledges financial support by the BMBF project ATIQ.

\bibliography{references}
\end{document}